\documentclass[3p, a4paper, 11pt]{elsarticle}

\usepackage{xcolor}
\usepackage{tikz}
\usepackage{layouts}
\biboptions{numbers,sort&compress}
\usepackage{algorithm}
\usepackage{algpseudocode}
\usepackage{gensymb}
\usepackage[unicode=true,colorlinks=true,pdfborder={0 0 0}]{hyperref}
\hypersetup{
    bookmarks=true,         
    bookmarksopen=true,     
    pdftoolbar=true,        
    pdfmenubar=true,        
    pdffitwindow=true,      
    pdfstartview={FitH},    
    pdfnewwindow=true,      
}
\usepackage{graphicx}
\usepackage{caption}
\usepackage{subcaption}
\usepackage{wrapfig}
\usepackage{multirow,makecell}
\usepackage{longtable}
\usepackage{booktabs}
\usepackage{tabularx}
\usepackage{setspace}
\usepackage{float}
\usepackage{enumerate}
\usepackage{enumitem}

\usepackage[amsthm,thmmarks]{ntheorem}
\usepackage{siunitx}
\usepackage{lineno}

\usepackage{math-symbols}
\usepackage{cases}

\graphicspath{{./images/}}

\theoremstyle{definition}
\newtheorem{theorem}{Theorem}
\newtheorem{proposition}[theorem]{Proposition}
\theoremstyle{definition}

\newtheorem{remark}{Remark}
\theoremstyle{definition}

\theoremstyle{definition}

\journal{XXX} 

\begin{document}

	\begin{frontmatter}
		\title{Data-driven balanced truncation of K-power bilinear systems}
		\author[home]{Xiaolong Wang\corref{cor1}}
		\ead{xlwang@nwpu.edu.cn}
		\author[home]{Biaolin Li}
		\author[coop]{Xiaoli Wang}
		\cortext[cor1]{Corresponding author}
		\address[home]{School of Mathematics and Statistics,
			Northwestern Polytechnical University, Xi'an 710072, China}
		\address[coop]{Xi'an Microelectronics Technology Institute, Xi'an 710065, 
		China}
		\begin{abstract}
			As a special type of bilinear systems, K-power bilinear systems possess 
			a special coupled structure along with nice properties in practice. In 
			this paper, we 
			investigate the data-driven counterpart of balanced truncation for 
			K-power systems. As the standard balanced truncation is performed based 
			on the subsystems of K-power systems, the main idea is to approximate 
			the quantities of each reduced subsystem with the evaluations of 
			transfer functions. We exploit the nice properties of Gramians for 
			K-power systems, 
			and establish the explicit relationship between the main quantities of 
			balanced truncation and the evaluation of transfer functions. As a 
			result, reduced models produced via balanced truncation can be assembled 
			approximately by the sample data of transfer functions, leading 
			to a data-driven balancing truncation method for K-power systems. An 
			advanced procedure is also provided to avoid the complex arithmetic 
			completely and produce real-valued reduced models. Two 
			numerical examples confirm the feasibility and effectiveness of the 
			proposed 
			method.
		\end{abstract}
		\begin{keyword}
			{Model order reduction, K-power systems, 
			Data-driven modeling, Balanced truncation, Gramians.}
		\end{keyword}
	\end{frontmatter}

\section{Introduction}\label{sec:introduction}

Large-scale dynamical systems arise frequently in all fields of engineering, such as 
in integrated circuits and micro systems, civil engineering, and 
Micro-Electro-Mechanical Systems \cite{Ramaswamy2000, Su1991, Benner2017book}. The 
fast simulation of such systems becomes an intractable task because of the 
unacceptably computational load. Moder order reduction (MOR) aims to replace a 
large-scale system with a lower order one so as to reduce the huge computational 
burden. The MOR techniques for linear systems have been well developed during the 
past years, mainly including two types of methods: methods based on the 
singular value decomposition (SVD) and methods based on Krylov subspace operators 
\cite{Antoulasbook2005}. MOR of nonlinear systems is generally much more difficult 
and challenging \cite{Baur2014,Pan2013,Chaturantabut2010}.

Bilinear systems are a special type of nonlinear systems, in which nonlinear terms 
arise from the product of the state and input. Due to the linear 
relationship with respect to the state and the input separately, bilinear 
systems are closely related to linear systems. Because a variety of nonlinear 
systems can be reformulated approximately as binlinear systems of high order via 
Carleman bilinearisation, it is a good medium to analyze the general nonlinear 
systems \cite{Gu2011,Rughbook1981}. There are some schemes dedicated to MOR of 
bilinear systems. The moment-matching method based on Krylov subspace techniques has 
been exploited to interpolate the multivariable transfer functions of bilinear 
systems \cite{Wang2012}. Another interpolatory strategy is 
also given in \cite {Flagg2015} to enforce multipoint interpolation of the 
underlying 
Volterra series of bilinear systems. Necessary conditions for the 
reduced 
order bilinear models to be $H_2$ optimal are given in \cite{Zhang2002}. An 
iterative algorithm is designed to yield a reduced model fulfilling these 
conditions, and it allows for an adaption of the successful iterative rational 
Krylov algorithm to bilinear systems \cite{Bennerbil2012}. The basic balanced 
truncation (BT) procedure is generalized initially to bilinear systems in 
\cite{Hsu1983}, 
and the interpretation of the input and output energies for balanced truncation is 
discussed in \cite{Bennerbil2011}. K-power bilinear systems is a special type of 
bilinear systems, and the input–output map of such systems is homogeneous with 
respect to the input of degree $k$. For this reason, they are also called degree-$k$ 
homogeneous systems \cite{Baiyat1993}. In \cite{Wang2014}, the moment-matching 
methods for K-power bilinear
systems are exploited to produce structure-preserved
reduced models from the perspective of bilinear systems and coupled 
systems, respectively, where the optimal $H_2$ MOR is also 
discussed. Alternatively, an approach building on the asymptotic 
expansion of K-power bilinear systems is reported in \cite{Qi2021}, where a desired 
number of expansion coefficients are preserved in the time domain. For BT method, 
the block diagonal structure of Gramians is proved in \cite{Baiyat1993}, and thereby 
a structure-preserved BT procedure is derived for K-power bilinear systems. 
Recently, a 
finite-time version of BT methods is used to enhance the approximation accuracy in a 
specific time interval for K-power bilinear systems \cite{Zhang2024}.

Recently, the standard BT procedure is executed in an 
approximate manner, and it entirely relies on the evaluations of transfer function 
and refrains from the intrusive access to any prescribed realization of the original 
system \cite{Gosea2022}. Because of the superior accuracy, the basic idea of such a
data-driven BT procedure has been applied to second-order systems with proportional 
damping, as well as linear systems with quadratic outputs \cite{Wang2025, 
Padhi2025}. Note that the data-driven MOR is an alternative approach for the 
simplification of large-scale systems, which based on the sample data in the time 
domain or frequency domain of the underlying system, instead of a specific 
mathematical model of the system. We refer the reader to \cite{Bhattacharjee2025, 
Ionita2014, Geelen2023, Huhn2023, Burohman2023} 
for more details on the data-driven approach. 

In this paper, we consider the data-driven counterpart of the BT procedure for 
K-power bilinear systems, and execute the nonintrusive BT based on the measurements 
of systems in the frequency domain. We start with the quadrature expression of 
Gramians for each subsystem of K-power bilinear systems, and employ a numerical 
quadrature rule to approximate Gramians in the frequency domain. As the 
controllability and observability Gramians of subsystems are coupled with a serial 
structure, but in a reverse order, the subsystems of K-power bilinear systems cannot 
be associated with the individual transfer functions directly from the perspective 
of bilinear systems. We devote to extracting the explicit expression for each 
Gramian 
in the framework of numerical quadrature, and derive a low-rank approximate 
decomposition for each Gramian. As a result, the main quantities involved in the 
intrusive BT can be approximated via the low-rank Gramians, and thereby can be 
calculated precisely via the measurements of the $k$-th transfer function. 
Besides, by choosing the quadrature nodes and weights in a symmetric manner with 
respect to the real axis in the numerical quadrature, we circumvent the complex 
arithmetic calculation in the execution, and provide a real-valued algorithm for the 
proposed data-driven BT procedure, leading to real-valued reduced models as 
well.

The paper is organized as follows. Section 2 introduces the preliminaries on K-power 
bilinear system. We start Section 3 with the standard BT, and approximate the 
Gramians via the numerical quadrature rule. A data-driven BT procedure is presented 
based on the evaluations of transfer function in real arithmetic. Numerical results 
are use to test our approach in 
Section 4. Finally, some conclusions are drawn in Section 5.

\section{Preliminaries}\label{sec:sec-2}
 
K-power bilinear systems are a special type of bilinear systems, which have the 
following state-space description 
\begin{equation}\label{S}
	\left\{
	\begin{aligned}
		\dot{x}(t) & =A x(t)+ \sum_{i=1}^p N_ix(t)u_i(t)+Bu(t), \\
		y(t) & =C x(t),
	\end{aligned}
     \right.
\end{equation}
where $x(t)\in\mathbb R^{n}$ is the state, $u(t)=[u_1(t), \cdots, 
u_p(t)]^\top\in\mathbb R^p$ is the input, and $y(t)\in\mathbb R^m$ is the output. 
We consider stable and minimal K-power systems, which can be reformulated 
as (\ref{S}) with special coefficient matrices 
\begin{equation}\label{S-str}
	A=\begin{bmatrix}
	A_1 & 0 & \cdots & 0\\
	0 & A_2 & \ddots & \vdots	\\
	\vdots & \ddots & \ddots &0\\
	0 & \cdots & 0 & A_k
	\end{bmatrix} \quad
    N_i=\begin{bmatrix}
    	0 & 0 & \cdots & 0\\
    	N_{1i} & 0 & \cdots & 0	\\
    	\vdots & \ddots & \ddots &\vdots\\
    	0 & \cdots & N_{(k-1)i} & 0
    \end{bmatrix} \quad
    B=\begin{bmatrix}
    	B_1\\
    	0\\
    	\vdots\\
    	0
    \end{bmatrix} \quad 
    C^\top=\begin{bmatrix}
    	0\\
    	\vdots\\
    	0\\
    	C_k^\top
    \end{bmatrix},
\end{equation}
where $A_j\in\mathbb R^{n_j\times n_j}$, $N_{ji}\in\mathbb R^{n_{j+1}\times n_j} 
(j=1, 2, \cdots, k-1)$, $A_k\in \mathbb R^{n_k\times n_k}$, $B_1\in\mathbb 
R^{n_1\times p}$, and $C_k\in\mathbb R^{m\times n_k}$. 
Accordingly, the state $x(t)$ is partitioned compatibly as $$x(t)=[x_1(t)^\top, 
x_2(t)^\top, \cdots, x_k(t)^\top]^\top,$$ where $x_j(t)\in\mathbb R^{n_j}$ 
for $j=1, 2, \cdots, k.$ For simplicity, we mainly focus on the single-input and 
single output (SISO) systems, i.e., $m=p=1$ in our 
discussion. However, all results 
obtained in this paper can be applied to multiple input and multiple output (MIMO) 
systems with some proper modification, as shown in Section 3.

Given the zero initial conditions, the $i$-th transfer function of (\ref{S}) 
reads                                    
\begin{equation}
	\begin{array}{l}
		H_{i}(s_{1}, \cdots, s_{i})= 
		C\Phi(s_i,A)N_1\Phi(s_{i-1},A) N_1 \cdots \Phi(s_1,A) B, 
	\end{array}
\end{equation}
where we use the function $\Phi(s, A)=(sI-A)^{-1}$. With the coefficient matrices 
(\ref{S-str}), one can verify directly that 
\begin{equation}\label{tran-f}
	\begin{split}
		H_{i}(s_{1}, \cdots, s_{i})&=0, \quad \text{for} \quad i\ne k\\
		H_{k}(s_{1}, \cdots, s_{k})&=C\Phi(s_k,A)N_1\cdots\Phi(s_{2},A) 
		N_1\Phi(s_1,A) B\\
		&=C_{k}\Phi(s_k,A_k) N_{(k-1)1}\cdots\Phi(s_2,A_2) N_{11}\Phi(s_1,A_1)
		B_{1}. 
	\end{split}
\end{equation}
This means that the dynamical behavior of K-power system can be determined 
completely via the $k$-th transfer function. In addition, it follows from 
(\ref{S-str}) that K-power systems can be rewritten as coupled systems
	\begin{equation}
		\left\{\begin{array}{l}
			\dot{x}_{1}(t)=A_{1} x_{1}(t)+B_{1} u(t), \\
			\dot{x}_{2}(t)=A_{2} x_{2}(t)+N_{11} x_{1}(t) u(t), \\
			\cdots, \\
			\dot{x}_{k}(t)=A_{k} x_{k}(t)+N_{(k-1)1}x_{k-1}(t) u(t), \\
			y(t)=C_{k} x_{k}(t). 
		\end{array}\right.
	\end{equation}
The above nice properties of K-power systems facilitate a lot the derivation of 
data-driven BT in next section.

\section{Data-driven BT of K-power bilinear 
systems}\label{sec:sec-3}
BT has been extensively studied for various systems \cite{*}. We first give a 
brief review on the standard BT for K-power systems, and then employ the 
numerical quadrature rule to present a data-driven counterpart in this section. 

\subsection{BT for K-power bilinear systems}	
Gramians play an important role in the standard BT procedure. Controllability 
Gramian $P$ and 
observability Gramians $Q$ of bilinear systems can be obtained by solving the 
following 
generalized Lyapunov equations 
\begin{equation}
	\begin{split}
		AP+PA^\top +N_1PN_1^\top+BB^\top=0,\\
		A^\top Q+QA+N_1^\top QN_1+C^\top C=0,
	\end{split}
\end{equation}
respectively. In the setting of K-power bilinear systems, it follows from 
(\ref{S-str}) that 
	\begin{equation}\label{PQ}
		P  =\operatorname{diag}\left[P_{11}, P_{22}, \cdots, P_{kk}\right],
		Q  =\operatorname{diag}\left[Q_{11}, Q_{22}, \cdots, Q_{kk}\right],
	\end{equation}
	where $P_{jj} $ and $ Q_{jj}$ solve the following Lyapunov 
	equations
	\begin{equation}
		\begin{array}{l}
			A_{1} P_{11}+P_{11} A_{1}^{\top}+B_{1} B_{1}^{\top}=0, \\
			A_{j} P_{j j}+P_{j j} A_{j}^ {\top}+ N_{(j-1)1} P_{(j-1)(j-1)} 
			N_{(j-1)1}^{\top}=0, 
			\quad j=2,3, \cdots, k
		\end{array}
	\end{equation}
	\begin{equation}
		\begin{array}{l}
			A_{k}^{\top} Q_{k k}+Q_{k k} A_{k}+C_{k}^{\top} C_{k}=0, \\
			A_{j}^{\top} Q_{jj}+Q_{jj} A_{j}+ N_{j1}^{\top} Q_{(j+1)(j+1)} N_{j1}=0, 
			\quad 
			j=k-1,k-2, \cdots, 1
		\end{array}
	\end{equation}
which implies that Gramians of K-power systems are available by solving the standard 
Lyapunov equations. Note that for stable K-power systems, $A_j$ are Hurwitz matrices 
for $j=1, 2, \cdots, k$, and each Lyapunov equation mentioned above has a unique 
solution. 

Once $P, Q$ have been determined, the balanced realization of (\ref{S}) 
can be obtained by applying a balancing matrix $T$. However, in order to perform a 
structure-preserving model reduction, a special balancing matrix 
$T=\mathrm{diag}\{T_1, T_2, \cdots, T_k\}$ with a 
block-diagonal structure is designed for BT of K-power systems. It is equivalent 
to performing model reduction from the subsystem point-of-view, that is, $T_j$ is 
designed to make the controllability and observability Gramians of subsystems equal 
and diagonal 
\begin{equation}
	\hat P_{jj}=\hat Q_{jj}=\hat \Sigma_j  \quad \text{for} \quad j=1, 2, \cdots, k.
\end{equation}
The singular values, diagonal elements of $\hat \Sigma_j$, can be used to determine 
the important modes and reduced models accordingly. Algorithm 1 summarizes the main 
steps of BT procedure for K-power bilinear systems. As the key quantities in step 2 
and step 4 of Algorithm 1 can be well approximated via the evaluation of transfer 
function in the frequency domain, a data-driven approach will be presented in the 
next subsection. 

\begin{algorithm}[htb]
	\caption{BT for K-power bilinear systems [6]}
	\label{bt-K-bilinear}
	\begin{algorithmic}[1]
		\Require
		System matrices $A_{j} \in \mathbb{R}^{n_{j} \times n_{j}}, N_{j1} \in 
		\mathbb{R}^{n_{j+1} \times n_{j}}$ $  B_{1} \in \mathbb{R}^{n_{1}}$  and $ 
		C_{k} \in \mathbb{R}^{1 \times n_{k}}$. 
		\Ensure
		Reduced system matrices $\hat A_{j}$, $\hat N_{j1}$, $\hat B_1$, $\hat C_k$.
		\State Compute the square factors $P_{j j}={L}_{j j}{L}_{j 
		j}^{\top}$, $Q_{jj}={R}_{jj}{R}_{j j}^{\top}$, and pick a 
		truncation 
		index $r_j$ for $j=1, 2, \cdots, k$.
		\State Compute SVD of ${R}^{\top}_{jj} 
		{L}_{jj}$ with the following partitioned form 
        \begin{equation*}
			{R}^{\top}_{jj} {L}_{jj}=\left[\begin{array}{ll}
				{U}_{j1} & {U}_{j2}
			\end{array}\right]\left[\begin{array}{ll}
				{S}_{j1} & \\
				&{S}_{j2}
			\end{array}\right]\left[\begin{array}{c}
				{Y}_{j1}^{\top} \\
				{Y}_{j2}^{\top}
			\end{array}\right],
	    \end{equation*}
        where ${S}_{j1}\in \mathbb R^{r_j\times r_j}$ and ${S}_{j2}\in \mathbb 
        R^{(n_j-r_j)\times (n_j-r_j)}$. 
		\State Assemble the projection matrices for each subsystem
		\begin{equation*}	
			{V}_{j}={L}_{jj} {Y}_{j1} {S}_{j1}^{-1 / 2} \quad
			\text {and} \quad {W}_{j}^{\top}={S}_{j1}^{-1 / 
			2}{U}_{j1}^{\top}{R}_{jj}^{\top}. 			
		\end{equation*}	
		\State Construct the balanced state-space matrices
		\begin{equation*}	
			\begin{array}{cc}
				\hat{N}_{j1}={W}_{j+1}^{\top} {N}_{j1} 
				{V}_{j}={S}_{(j+1)1}^{-1/ 
				2}{U}_{(j+1)1}^{\top}\left({R}_{(j+1)(j+1)}^{\top}N_{j1}{L}_{jj} 
				\right) {Y}_{j1}
				{S}_{j1}^{-1/2}, \\ 
				\hat{A}_{j}={W}_{j}^{\top} {A}_{j} 
				{V}_{j}={S}_{j1}^{-1/ 
					2}{U}_{j1}^{\top}\left({R}_{jj}^{\top}A_{j}{L}_{jj} \right) 
					{Y}_{j1}
				{S}_{j1}^{-1/2}, \\
				{\hat B}_{1}={W}_{1}^{\top} {B_1}={S}_{11}^{-1 / 
				2} {U}_{11}^{\top}\left({R}^{\top}_{11} 
				{B}_{1}\right), \\ 
				{\hat C}_{k}={C_k} 
				{V}_{k}=\left({C_k}L_{kk}\right) {Y}_{k1} 
				{S}_{k1}^{-1 / 2}. 
			\end{array}
		\end{equation*}	
		
	\end{algorithmic}
\end{algorithm}

\subsection{Quadrature-based approximation to main quantities via the sample 
data}\label{sec:sec-3.1}
	We consider the Gramians given in (\ref{PQ}). As $P_{jj}$ solves the standard 
	Lyapunov equation, it has the following quadrature-based definition in the time 
	domain 
	\begin{equation}
		\begin{split}
		P_{1 1}&=\int_{0}^{\infty} e^{{A_{1}} t}  B_{1} {B_{1}}^{\top}  
		e^{{A}^{\top}_{1}  
		t} d t,\\
		P_{j j}&=\int_{0}^{\infty} e^{{A_{j}} t} N_{(j-1)1} P_{(j-1)(j-1)} 
		N_{(j-1)1}^{\top} 
		e^{{A}^{\top}_{1} t} d t, 
		\end{split}	
	\end{equation}
for $j=2, 3, \cdots, k$. Let $\rm{i} = \sqrt{-1}$. The Parseval's theorem leads to 
an 
equivalent expression in the frequency domain as follows 
	\begin{equation}
		\begin{split}
		{P}_{11}=\frac{1}{2 \pi} \int_{-\infty}^{\infty}\Phi(\mathrm{i} 
		\omega_1,A_1)B_{1} {B}^{\top}_{1}\Phi(\mathrm{-i}\omega_1,A_1^{\top}) 
		d\omega_1,\\
		{P}_{jj}=\frac{1}{ 2\rm{\pi}} \int_{-\infty}^{\infty} \Phi(\mathrm{i} 
		\omega_j,A_j)N_{(j-1)1}{P}_{(j-1)(j-1)}N_{(j-1)1}^{\top}\Phi(-\mathrm{i} 
		\omega_j,A_j^{\top})d\omega_j
		\end{split}
	\end{equation}
for $j=2, 3, \cdots, k$. We now adopt the numerical quadrature rule to calculate 
$P_{jj}$ approximately in the frequency domain. Specifically, we have 
	\begin{equation*}	
		{P}_{11} \approx \widetilde{{P}}_{11}=\sum_{i_1=1}^{\gamma_1} 
		\rho_{i_1,1}^{2}\Phi(\mathrm{i} \lambda_{i_1,1},A_1){B}_{1} 
		{B}^{\top}_{1}\Phi(-\mathrm{i} \lambda_{i_1,1},A_1^{\top}),
	\end{equation*}	
where $\lambda_{i_1,1}$ and $\rho_{i_1,1}^2$ represent the numerical quadrature 
nodes 
and  
weights, respectively, and $\gamma_1$ is the total number of quadrature nodes. 
	With the same spirit, ${P}_{22}$ can be approximated as
\begin{equation*}
	\begin{split}
		P_{22}&\approx \sum_{i_2=1}^{\gamma_2} 
		\rho_{i_2,2}^{2}\Phi(\mathrm{i} \lambda_{i_2,2}, {A}_{2})N_{11} 
		P_{11} 
		N_{11}^{\top} \Phi(-\mathrm{i} \lambda_{i_2,2}, {A}^{\top}_{2})\\
		&\approx \sum_{i_2=1}^{\gamma_2} \sum_{i_1=1}^{\gamma_1}
		\rho_{i_2,2}^{2}\rho_{i_1,1}^{2}\Phi(\mathrm{i} \lambda_{i_2,2}, 
		{A}_{2})N_{11} 
		\Phi(\mathrm{i} \lambda_{i_1,1},A_1){B}_{1} 
		{B}^{\top}_{1}\Phi(-\mathrm{i} \lambda_{i_1,1},A_1^{\top}) 
		N_{11}^{\top} \Phi(-\mathrm{i} \lambda_{i_2,2}, {A}^{\top}_{2})\\		
		&=\widetilde{P}_{22}.
	\end{split}
\end{equation*}	
Likewise, for $j\leq k$, the numerical quadrature rule leads to 
\begin{equation*}
	P_{jj}\approx \widetilde{P}_{jj}=\sum_{i_j=1}^{\gamma_j} 
	\rho_{i_j,j}^{2}\Phi(\mathrm{i} \lambda_{i_j,j}, {A}_{j})N_{(j-1)1} 
	\widetilde P_{(j-1)(j-1)} 
	N_{(j-1)1}^{\top} \Phi(-\mathrm{i} \lambda_{i_j,j}, {A}^{\top}_{j}).
\end{equation*}		
Note that there are $\mathcal{N}_j=\gamma_1\cdots\gamma_j$ quadrature nodes in the 
approximation to $P_{jj}$. By defining the square-root factor  $\widetilde{L}_{jj} 
\in \mathbb{C}^{n_{j} \times \mathcal{N}_{j}} $ for $j=1, 2, \cdots, k$ as follows
\begin{equation}\label{L-factor}
	\begin{split}
		\widetilde{L}_{11}&=\left[\begin{array}{lll}
			\rho_{1, 1}\Phi(\mathrm{i} \lambda_{1,1},A_1){B}_{1}  & 
			\cdots & \rho_{\gamma_1,1}\Phi(\mathrm{i} 
			\lambda_{\gamma_1,1},A_1){B}_{1} 
		\end{array}\right] \in \mathbb{C}^{n_{1} \times \mathcal{N}_{1} }, \\
	\widetilde{L}_{jj}&=\left[\begin{array}{lll}
		\rho_{1,j}\Phi(\mathrm{i} 
		\lambda_{1,j},A_j)N_{(j-1)1}\widetilde{L}_{(j-1)(j-1)} & 
		\cdots & \rho_{\gamma_j,j}\Phi(\mathrm{i} 
		\lambda_{\gamma_j,j},A_j)N_{(j-1)1}\widetilde{L}_{(j-1)(j-1)}
	\end{array}\right]\in\mathbb{C}^{n_{j} \times\mathcal{N}_{j}},   	
	\end{split}
\end{equation}
Gramians of subsystems have the approximation  
${P}_{jj}\approx \widetilde{P}_{jj}=\widetilde{L}_{jj} \widetilde{L}_{jj}^{*}. $

Similarly, the observability Gramians $Q_{jj}$ of each subsystem 
have the following expression 
	\begin{equation*}
		\begin{split}
			{Q}_{jj}&=\frac{1}{2 \pi} \int_{-\infty}^{\infty} 
			\Phi(\mathrm{-i}\omega_j,A_j^{\top})N_{j1}^{\top} Q_{(j+1)(j+1)} 
			N_{j1}\Phi(\mathrm{i} 
			\omega_k,A_k)
			d\omega_k,\\
			{Q}_{kk}&=\frac{1}{2 \pi} \int_{-\infty}^{\infty} 
			\Phi(\mathrm{-i}\omega_k,A_k^{\top}) C_{k}^\top C_{k}\Phi(\mathrm{i} 
			\omega_k,A_k)
			d\omega_k
		\end{split}			
	\end{equation*}
for $j=1, 2, \cdots, k. $ With the quadrature nodes $\mu_{i_j,j}$ and weights 
$\phi_{i_j,j}$, we have the approximation to Gramians of each subsystem  
\begin{equation*}	
	\begin{split}
	{Q}_{kk} &\approx \widetilde{{Q}}_{kk}=\sum_{i_k=1}^{\gamma_k} 
	\phi_{i_k,k}^{2}\Phi(\mathrm{-i} \mu_{i_k,k},A_k^{\top}){C}_{k}^{\top} 
	{C}_{k}\Phi(\mathrm{i} \mu_{i_k,k},A_k),\\
		Q_{jj}&\approx \widetilde{Q}_{jj}=\sum_{i_j=1}^{\gamma_j} 
		\phi_{i_j,j}^{2}\Phi(\mathrm{-i} \mu_{i_j,j},A_j^{\top})N_{j1}^{\top} 
		\widetilde 
		Q_{(j+1)(j+1)} 
		N_{j1}\Phi(\mathrm{i} \mu_{i_j,j},A_j).
	\end{split}
\end{equation*}			
Note that there are $\bar{\mathcal{N}}_j=\gamma_k\cdots\gamma_j$ quadrature 
nodes in 
the approximation to $Q_{jj}$. The approximation has the expression $\widetilde 
{Q}_{jj}=\widetilde{R}_{jj}\widetilde{R}^{\top}_{jj}$ along with the square-root 
factor 
\begin{equation}\label{R-factor}
	\begin{split}	
		\widetilde{R}^{\top}_{kk}&=\left[\begin{array}{c}
			\phi_{1,k} C_{k}\Phi(\mathrm{i} \mu_{1,k},A_k)\\
			\phi_{2,k} C_{k}\Phi(\mathrm{i} \mu_{2,k},A_k)\\  
			\vdots \\
			\phi_{\gamma_k,k} C_{k}\Phi(\mathrm{i} \mu_{\gamma_k,k},A_k)\\ 
		\end{array}\right] \in \mathbb{C}^{\mathcal{\gamma}_{k} \times n_{k}}.\\
		\widetilde{R}^{\top}_{jj}&=\left[\begin{array}{c}
			\phi_{1,j}\widetilde{R}^{\top}_{(j+1)(j+1)}N_{j1}\Phi(\mathrm{i} 
			\mu_{1,j},A_j)\\
			\phi_{2,j}\widetilde{R}^{\top}_{(j+1)(j+1)}N_{j1}\Phi(\mathrm{i} 
			\mu_{2,j},A_j) \\  
			\vdots \\
			\phi_{\gamma_j,j}\widetilde{R}^{\top}_{(j+1)(j+1)}N_{j1}\Phi(\mathrm{i} 
			\mu_{\gamma_j,j},A_j)\\ 
		\end{array}\right] \in \mathbb{C}^{\bar{\mathcal{N}}_j \times n_{j}}.
	\end{split}
\end{equation}	

\begin{proposition}\label{pro-1}
Let $ \widetilde{L}_{jj}$ and $\widetilde{R}_{jj}$ be defined in (\ref{L-factor}) 
and (\ref{R-factor}) for $1 \leq j \leq k$. Define the matrices 
$\widetilde{\mathbb{U}}_{j}=\widetilde{R}_{jj}^{\top}\widetilde{L}_{jj}$ and  
$\widetilde{\mathbb{A}}_{j}=\widetilde{R}_{jj}^{\top}{A}_{j}\widetilde{L}_{jj}.$
For $1 \leq h\leq \mathcal{N}_{j}, 1 \leq l\leq \mathcal{\bar N}_{j}$, the $(h,l)$  
element of $\widetilde{\mathbb{U}}_{j}$ and 
$\widetilde{\mathbb{A}}_{j}$ can be expressed via the evaluations of the $k$-th 
transfer function as follows 
\begin{equation}\label{U_ele}
		\begin{split}
		\widetilde{\mathbb{U}}_{j}^{(h,l)}=&-\frac{\delta_{hl}H_k(\mathrm{i}\mu_{i_{k},k},
		 \cdots, \mathrm{i}\mu_{i_{j+1},j+1}, 
			\mathrm{i} \lambda_{i_{j},j}, \cdots, 
			\mathrm{i}\lambda_{i_1,1})}{\mathrm 
		i(\lambda_{i_j,j}-\mu_{i_j,j})}\\
		&+\frac{\delta_{hl}H_k(\mathrm{i}\mu_{i_{k},k}, \cdots, 
			\mathrm{i}\mu_{i_{j},j}, 
			\mathrm{i} \lambda_{i_{j-1},j-1}, \cdots, \mathrm{i} 
			\lambda_{i_1,1})}{\mathrm 
			i(\lambda_{i_j,j}-\mu_{i_j,j})},
	   \end{split}
\end{equation}
\begin{equation}\label{A_ele}
	\begin{split}	
		\widetilde{\mathbb{A}}_{j}^{(h,l)}=
		 &-\frac{\delta_{hl}\lambda_{i_j,j}H_k(\mathrm{i}\mu_{i_{k},k},
		 	\cdots, \mathrm{i}\mu_{i_{j+1},j+1}, 
		 	\mathrm{i} \lambda_{i_{j},j}, \cdots, 
		 	\mathrm{i}\lambda_{i_1,1})}{\lambda_{i_j,j}-\mu_{i_j,j}}\\
		 &+\frac{\delta_{hl}\mu_{i_j,j}H_k(\mathrm{i}\mu_{i_{k},k}, \cdots, 
		 	\mathrm{i}\mu_{i_{j},j}, 
		 	\mathrm{i} \lambda_{i_{j-1},j-1}, \cdots, \mathrm{i} 
		 	\lambda_{i_1,1})}{\lambda_{i_j,j}-\mu_{i_j,j}}, 
	\end{split}
\end{equation}
where the constant 
$\delta_{hl}=\phi_{i_k,k}\cdots\phi_{i_j,j}\rho_{i_j,j}\cdots\rho_{i_1,1}.$
\end{proposition}

\begin{proof}
Without loss of generality, the $h$-th row of $\widetilde R^\top_{jj}$ reads 
\begin{equation*}
	\phi_{i_k,k}\dots\phi_{i_j,j} C_{k}\Phi(\mathrm{i} 
	\mu_{i_k,k},A_k)N_{(k-1)1}\Phi(\mathrm{i} \mu_{i_{k-1},k-1},A_{k-1})\cdots 
	N_{j1}\Phi(\mathrm{i} \mu_{i_{j},j},A_{j}), 
\end{equation*}
and the $l$-th column of $\widetilde L_{jj}$ reads 
\begin{equation*}
	\rho_{i_1,1}\dots\rho_{i_j,j}\Phi(\mathrm{i} 
	\lambda_{i_{j},j},A_{j})N_{(j-1)1}\dots\Phi(\mathrm{i} 
	\lambda_{i_{2},2},A_{2})N_{11}\Phi(\mathrm{i} \lambda_{i_{1},1},A_{1})B_{1}.
\end{equation*}
Consequently, the $(h,l)$ element of $\widetilde{\mathbb{U}}_{j}$ has the following 
expression 
\begin{equation}\label{equality_U}
    \widetilde{\mathbb{U}}_{j}^{(h,l)}=\delta_{hl}C_{k}\Phi(\mathrm{i} 
    \mu_{i_k,k},A_k)\cdots 
    N_{j1}\Phi(\mathrm{i} \mu_{i_{j},j},A_{j})\Phi(\mathrm{i} 
    \lambda_{i_{j},j},A_{j})N_{(j-1)1}\dots\Phi(\mathrm{i} 
    \lambda_{i_{1},1},A_{1})B_{1}. 
\end{equation}	
For any square matrix, $X$, and any $a, b\in \mathbb C$ that are not eigenvalues of 
$X$, there exists the identity 
\begin{equation*}
	(aI-X)^{-1}(bI-X)^{-1}=\frac{1}{a-b}\left((bI-X)^{-1}-(aI-X)^{-1}\right).
\end{equation*}
It follows from the above equality that 
\begin{equation*}
	\Phi(\mathrm{i} \mu_{i_{j},j},A_{j})\Phi(\mathrm{i} 
	\lambda_{i_{j},j},A_{j})=\frac{1}{\mathrm i(\mu_{i_{j},j}-\lambda_{i_{j},j})}
	\left(\Phi(\mathrm{i} 
	\lambda_{i_{j},j},A_{j})-\Phi(\mathrm{i} \mu_{i_{j},j},A_{j})\right). 
\end{equation*}
Substituting the above equality into \eqref{equality_U} leads to 
\begin{equation*}
	\widetilde{\mathbb{U}}_{j}^{(h,l)}=\frac{\delta_{hl}}{\mathrm 
	i(\mu_{i_{j},j}-\lambda_{i_{j},j})}C_{k}\Phi(\mathrm{i} 
	\mu_{i_k,k},A_k)\cdots 
	N_{j1}\left(\Phi(\mathrm{i} 
	\lambda_{i_{j},j},A_{j})-\Phi(\mathrm{i} 
	\mu_{i_{j},j},A_{j})\right)N_{(j-1)1}\dots\Phi(\mathrm{i} 
	\lambda_{i_{1},1},A_{1})B_{1}. 
\end{equation*}
As a result, one can get \eqref{U_ele} by using the definition in \eqref{tran-f}. 

Similarly, there holds the identity 
\begin{equation*}
	(aI-X)^{-1}X(bI-X)^{-1}=\frac{1}{a-b}\left(b(bI-X)^{-1}-a(aI-X)^{-1}\right).
\end{equation*}
It follows that 
\begin{equation*}
	\Phi(\mathrm{i} \mu_{i_{j},j},A_{j})A_j\Phi(\mathrm{i} 
	\lambda_{i_{j},j},A_{j})=\frac{1}{\mu_{i_{j},j}-\lambda_{i_{j},j}}
	\left(\lambda_{i_{j},j}\Phi(\mathrm{i} 
	\lambda_{i_{j},j},A_{j})-\mu_{i_{j},j}\Phi(\mathrm{i} 
	\mu_{i_{j},j},A_{j})\right).  
\end{equation*}
Then, one can validate \eqref{A_ele} readily via the above equality in the same way. 
\end{proof}

Furthermore, we define the matrices 
$$\widetilde{\mathbb{N}}_{j}=\widetilde{R}_{(j+1)(j+1)}^{\top}{N}_{j1}\widetilde{L}_{jj},
\widetilde{\mathbb{B}}_{1}=\widetilde{R}_{11}^{\top}B_{1}, 
\widetilde{\mathbb{C}}_{k}=C_{k}\widetilde{L}_{kk}.$$ 
The $(h,l)$ element of $\widetilde{\mathbb{N}}_{j}$ has the expression 
\begin{equation}\label{N-exp}
\widetilde{\mathbb{N}}_{j}^{(h,l)}=\frac{\delta_{hl}}{\phi_j}
H_k(\mathrm{i}\mu_{i_{k},k},
\cdots, \mathrm{i}\mu_{i_{j+1},j+1}, 
\mathrm{i} \lambda_{i_{j},j}, \cdots, 
\mathrm{i}\lambda_{i_1,1}). 
\end{equation}
The $h$-th element of $\widetilde{\mathbb{B}}_{1}$ reads
\begin{equation}\label{B-exp}
	\widetilde{\mathbb{B}}_{1}^{(h)}=\phi_{i_k,k}\dots\phi_{i_1,1}
	H_k(\mathrm{i}\mu_{i_{k},k},
	\cdots, \mathrm{i}\mu_{i_{j+1},j+1}, 
	\mathrm{i} \mu_{i_{j},j}, \cdots, 
	\mathrm{i}\mu_{i_1,1}). 
\end{equation}
The $l$-th element of $\widetilde{\mathbb{C}}_{k}$ reads
\begin{equation}\label{C-exp}
	\widetilde{\mathbb{C}}_{k}^{(l)}=\rho_{i_k,k}\dots\rho_{i_1,1}
	H_k(\mathrm{i}\lambda_{i_{k},k},
	\cdots, \mathrm{i}\lambda_{i_{j+1},j+1}, 
	\mathrm{i} \lambda_{i_{j},j}, \cdots, 
	\mathrm{i}\lambda_{i_1,1}). 
\end{equation}

\begin{remark}
	We use the notations 
	$h(j)=\sum_{d=j}^{k-1}(i_d-1)\gamma_k\cdots\gamma_{d+1}+i_k$ 
	and $l(j)=\sum_{d=2}^{j}(i_d-1)\gamma_{1}\cdots\gamma_{d-1}+i_1$ for 
	$j=1,2,\cdots, k$. It follows 
	from the definition of $\widetilde{R}^{\top}_{jj}$ and $\widetilde{L}_{jj}$ that 
	$h=h(j)$ and $l=l(j)$ in the proof of 
	Proposition \ref{pro-1}. For the $(h,l)$ element of 
	$\widetilde{\mathbb{N}}_{j}$, we have $h=h(j+1), l=l(j)$, and $h=h(1), l=l(k)$ 
	in the definition of $\widetilde{\mathbb{B}}_{1}$ and 
	$\widetilde{\mathbb{C}}_{k}$.  
\end{remark}

\begin{remark}
	For brevity, we use the same number of quadrature nodes in 
	the approximations $\widetilde{P}_{jj}$ and $\widetilde{Q}_{jj}$ about 
	$\omega_j$ in above 
	discussion. In general, one can choose a different number of quadrature 
	nodes for the approximation to $P_{jj}$ and $Q_{jj}$. 
\end{remark}

Now we are ready to present the data-driven BT for K-power bilinear systems. With 
the approximation 
$\widetilde{\mathbb{U}}_{j}=\widetilde{R}_{jj}^{\top}\widetilde{L}_{jj}\approx{R}^{\top}_{jj}
 {L}_{jj}$, the step 2 of Algorithm \ref{bt-K-bilinear} can be executed directly. 
 Alongside the approximations $\widetilde{\mathbb A}_j, \widetilde{\mathbb N}_j, 
 \widetilde{\mathbb B}_1$ and $\widetilde{\mathbb C}_k$, the whole algorithm can be 
 implemented based on 
 the evaluations of transfer function at some sampling points for a given numerical 
 quadrature rule. The main steps of the proposed method are summarized in Algorithm 
 \ref{data_BT}. 
 
\begin{algorithm}[htb]
	\caption{Data-driven BT of K-power bilinear systems}
	\label{data_BT}
	\begin{algorithmic}[1]
		\Require
		The weights $\rho_{i_j,j}, \mu_{i_j, j}$ and quadrature nodes 
		$\mathrm{i}\lambda_{i_j, j}, \mathrm{i}\phi_{i_j, j}$ for $1\le i_j \le 
		\gamma_j$ and $j=1, 
		2, \cdots, k$; The evaluation of the transfer function $H_k$ at the 
		nodes, and the 
		index $1\le r_j \le \min\{\mathcal{N}_{j},\bar{\mathcal{N}}_j\}$. 
		\Ensure
		Reduced models  $\widetilde{A}_j, \widetilde{N}_{j1}, \widetilde{B}_1, 
		\widetilde{C}_k$.
		
		\State  Assemble the main terms $\widetilde{\mathbb U}_j, 
		\widetilde{\mathbb A}_j, \widetilde{\mathbb N}_j, \widetilde{\mathbb 
			B}_1, \widetilde{\mathbb C}_k$ via \eqref{U_ele}, \eqref{A_ele}, 
		\eqref{N-exp}, \eqref{B-exp} and \eqref{C-exp}, respectively. 
		
		\State Compute the SVD of the matrix $\widetilde{\mathbb U}_j$ for $j=1, 
		2, \cdots, k$
		\begin{equation*}
			\widetilde{\mathbb U}_j=\left[\begin{array}{ll}
				{\widetilde U}_{j1} & {\widetilde U}_{j2}
			\end{array}\right]\left[\begin{array}{ll}
				{\widetilde S}_{j1} & \\
				&{\widetilde S}_{j2}
			\end{array}\right]\left[\begin{array}{c}
				{\widetilde Y}_{j1}^{\mathrm H} \\
				{\widetilde Y}_{j2}^{\mathrm H}
			\end{array}\right],
		\end{equation*}	
		where $\widetilde{S}_{j1}\in \mathbb R^{r_j \times r_j}$. 
		
		\State The reduced models are given by 
		\begin{equation*}
			\begin{split}
				\widetilde{N}_{j1}=\widetilde{S}_{(j+1)1}^{-1/2}\widetilde{U}_{(j+1)1}^{\mathrm
					H}\widetilde{\mathbb 
					N}_j\widetilde{Y}_{j1}\widetilde{S}_{j1}^{-1/2},\quad
				&\widetilde{A}_j=\widetilde{S}_{j1}^{-1/2}\widetilde{U}_{j1}^{\mathrm
					H}\widetilde{\mathbb 
					A}_j\widetilde{Y}_{j1}\widetilde{S}_{j1}^{-1/2},\\
				\widetilde{B}_1=\widetilde{S}_{11}^{-1/2}\widetilde{U}_{11}^{\mathrm
					H}\widetilde{\mathbb B}_1,\quad
				&\widetilde{C}_k=\widetilde{\mathbb 
					C}_k\widetilde{Y}_{k1}\widetilde{S}_{k1}^{-1/2}. 
			\end{split}
		\end{equation*}	
	\end{algorithmic}
\end{algorithm}

\subsection{Execution of the data-driven BT}

In practice, the dynamical systems are defined typically by the real-valued 
matrices, which ensures a real-valued output for a given input function and a 
initial condition. However, Algorithm \ref{data_BT} results in dynamical systems 
with complex-valued matrices in general because of the evaluations of transfer 
function along the imaginary axis. In what follows, we provide an advanced 
procedure, which avoids the complex arithmetic completely and results in real-valued 
reduced models. 

Let the number of quadrature nodes for all variables $\omega_j$ be even, that is 
$\gamma_j$ is an even number for $j=1, 2, \cdots, k$. We assume that the quadrature 
nodes and the weights are distributed symmetrically along the real axis, i.e., 
\begin{equation*}
	\begin{split}
	\lambda_{1,j}< \lambda_{2,j}< \cdots< \lambda_{\gamma_{j/2},j}< 0 < 
	\lambda_{\gamma_{j/2+1},j} 
	< \cdots < \lambda_{\gamma_j-1,j} <\lambda_{\gamma_j,j},\\
	\mu_{1,j}< \mu_{2,j}<\cdots< \mu_{\gamma_{j/2},j}< 0 < \mu_{\gamma_{j/2+1},j} < 
	\cdots < 
	\mu_{\gamma_j-1,j}<\mu_{\gamma_j,j},
	\end{split}
\end{equation*}
such that $\lambda_{i_j, j}=-\lambda_{\gamma_j/2+i_j, j}, \mu_{i_j, 
j}=-\mu_{\gamma_j/2+i_j, j}$ and the associated weights $\rho_{i_j, 
j}=\rho_{\gamma_j/2+i_j, j}, \phi_{i_j, j}=\phi_{\gamma_j/2+i_j, j}$ for $i_j=1, 2, 
\cdots, \gamma_j/2$. We rearrange rows of the factor 
$\widetilde{R}_{jj}^{\top}$ such that all rows are ordered in pairs of 
conjugation, that is, if one row of $\widetilde{R}_{jj}^{\top}$ is as follows 
\begin{equation*}
	\phi_{i_k,k}\dots\phi_{i_j,j} C_{k}\Phi(\mathrm{i} 
	\mu_{i_k,k},A_k)N_{(k-1)1}\Phi(\mathrm{i} \mu_{i_{k-1},k-1},A_{k-1})\cdots 
	N_{j1}\Phi(\mathrm{i} \mu_{i_{j},j},A_{j}), 
\end{equation*}
the next one is 
\begin{equation*}
	\phi_{i_k,k}\dots\phi_{i_j,j} 
	C_{k}\Phi(\overline{\mathrm{i}\mu_{i_k,k}},A_k)N_{(k-1)1}\Phi(\overline{\mathrm{i}\mu_{i_{k-1},k-1}},A_{k-1})
	\cdots N_{j1}\Phi(\overline{\mathrm{i} \mu_{i_{j},j}},A_{j}),
\end{equation*}
where $\bar s$ is the conjugation of the complex number $s$.
Similarly, the columns of the factor $\widetilde{L}_{jj}$ are ordered in the same 
manner. We partition the matrices $\widetilde{\mathbb U}_j$ and $\widetilde{\mathbb 
A}_j$ into 
$2\times 2$ blocks, $\widetilde{\mathbb U}_j^{(2)}, \widetilde{\mathbb A}_j^{(2)}$, 
which are 
compatible with the conjugate pairs for the factors $\widetilde{R}_{jj}^{\top}$ and 
$\widetilde{L}_{jj}$.

Recall the expression of $\widetilde{\mathbb U}_{j}^{(h,l)}$ and $\widetilde{\mathbb 
A}_{j}^{(h,l)}$ in (\ref{U_ele}) and (\ref{A_ele}). It is clear that 
$\widetilde{\mathbb U}_{j}^{(h,l)}$ and $\widetilde{\mathbb A}_{j}^{(h,l)}$ satisfy 
the following linear system 
\begin{equation}
	\left\{
	\begin{split}
		\mathrm{i}\lambda_{i_j,j}\widetilde{\mathbb 
		U}_{j}^{(h,l)}-\widetilde{\mathbb 
			A}_{j}^{(h,l)}&=\delta_{hl}H_k(\mathrm{i}\mu_{i_{k},k}, \cdots, 
		\mathrm{i}\mu_{i_{j},j}, 
		\mathrm{i} \lambda_{i_{j-1},j-1}, \cdots, \mathrm{i} 
		\lambda_{i_1,1}),\\
		\mathrm{i}\mu_{i_j,j}\widetilde{\mathbb U}_{j}^{(h,l)}-\widetilde{\mathbb 
			A}_{j}^{(h,l)}&=\delta_{hl}H_k(\mathrm{i}\mu_{i_{k},k}, \cdots, 
			\mathrm{i} 
		\mu_{i_{j+1},j+1},\mathrm{i}\lambda_{i_{j},j}, \cdots, \mathrm{i} 
		\lambda_{i_1,1}).
	\end{split}
	\right.
\end{equation}
Due to the conjugation pair in the rows and the columns of 
$\widetilde{R}_{jj}^{\top}$ and $\widetilde{L}_{jj}$, respectively, the $2\times 2$ 
blocks
\begin{equation*}
\widetilde{\mathbb U}_j^{(2)}=\left[\begin{array}{ll}
	\widetilde{\mathbb U}_j^{(h,l)} & \widetilde{\mathbb U}_j^{(h,l+1)}\\
	\widetilde{\mathbb U}_j^{(h+1,l)}&\widetilde{\mathbb U}_j^{(h+1,l+1)}
\end{array}\right], 
\widetilde{\mathbb A}_j^{(2)}=\left[\begin{array}{ll}
\widetilde{\mathbb A}_j^{(h,l)} & \widetilde{\mathbb A}_j^{(h,l+1)}\\
\widetilde{\mathbb A}_j^{(h+1,l)}&\widetilde{\mathbb A}_j^{(h+1,l+1)}
\end{array}\right]
\end{equation*}
satisfy the following linear system 
\begin{equation}\label{block-sys}
	\left\{
	\begin{split}
		\widetilde{\mathbb U}_j^{(2)}\left[\begin{array}{ll}
			\mathrm{i}\lambda_{i_j,j} & 0 \\
			0 &\overline{\mathrm{i}\lambda_{i_j,j}}
		\end{array}\right]-\widetilde{\mathbb A}_j^{(2)}=\delta_{hl}
	\left[\begin{array}{ll}
		 H_k^{(\mu,1)}& H_k^{(\mu,2)}\\
		 \overline{H_k^{(\mu,2)}}&\overline{H_k^{(\mu,1)}}
	\end{array}\right], \\
    \left[\begin{array}{ll}
    	\mathrm{i}\mu_{i_j,j} & 0 \\
    0 &\overline{\mathrm{i}\mu_{i_j,j}}
    \end{array}\right]\widetilde{\mathbb U}_j^{(2)}-\widetilde{\mathbb 
    A}_j^{(2)}=\delta_{hl}
    \left[\begin{array}{ll}
    	H_k^{(\lambda,1)}& H_k^{(\lambda,2)}\\
    	\overline{H_k^{(\lambda,2)}}&\overline{H_k^{(\lambda,1)}}
    \end{array}\right],
	\end{split}
	\right.
\end{equation}
where the notations are defined as follows 
\begin{equation*}
	\begin{split}
	H_k^{(\mu,1)}&=H_k(\mathrm{i}\mu_{i_{k},k}, \cdots, 
	\mathrm{i}\mu_{i_{j},j}, 
	\mathrm{i} \lambda_{i_{j-1},j-1}, \cdots, \mathrm{i} 
	\lambda_{i_1,1}), \\
	H_k^{(\mu,2)}&=H_k(\mathrm{i}\mu_{i_{k},k}, \cdots, 
	\mathrm{i}\mu_{i_{j},j}, 
	\mathrm{i} \overline{\lambda_{i_{j-1},j-1}}, \cdots, \mathrm{i} 
	\overline{\lambda_{i_1,1}}),\\
	H_k^{(\lambda,1)}&=H_k(\mathrm{i}\mu_{i_{k},k}, \cdots, 
	\mathrm{i}\mu_{i_{j+1},j+1}, 
	\mathrm{i} \lambda_{i_{j},j}, \cdots, \mathrm{i} 
	\lambda_{i_1,1}), \\
	H_k^{(\lambda,2)}&=H_k(\mathrm{i}\mu_{i_{k},k}, \cdots, 
	\mathrm{i}\mu_{i_{j+1},j+1}, 
	\mathrm{i} \overline{\lambda_{i_{j},j}}, \cdots, \mathrm{i} 
	\overline{\lambda_{i_1,1}}).
	\end{split}
\end{equation*}
Note that we use the property in the above that $H_{k}(\overline{s_{1}}, \cdots, 
\overline{s_{k}})=\overline{H_{k}(s_{1}, \cdots, s_{k})}$ for the given points $s_1, 
\cdots, s_k\in \mathbb C$, which can be validated directly by the definition 
(\ref{tran-f}). With the unitary matrix 
\begin{equation*}
	J=\frac{1}{\sqrt 2}\left[
	\begin{array}{cc}
		1&-\mathrm{i}\\
		1&\mathrm{i}
	\end{array}
	\right],
\end{equation*}
we define the transformation $\widetilde{\mathbb U}_j^{(2)\mathrm 
R}=J^*\widetilde{\mathbb U}_j^{(2)}J$ and $\widetilde{\mathbb A}_j^{(2)\mathrm 
R}=J^*\widetilde{\mathbb A}_j^{(2)}J$, where $*$ denotes the conjugate transpose for 
a given matrix. Taking $\widetilde{\mathbb U}_j^{(2)\mathrm 
R}$ and $\widetilde{\mathbb A}_j^{(2)\mathrm 
R}$ into (\ref{block-sys}) and performing some basic matrix manipulation lead to 
\begin{equation}\label{coup-sylvster}
	\left\{
	\begin{split}
		\widetilde{\mathbb U}_j^{(2)\mathrm R}\Delta_{\lambda}-\widetilde{\mathbb 
		A}_j^{(2)\mathrm R}=\delta_{hl}
		\left[\begin{array}{ll}
			\mathrm{Re}\left(H_k^{(\mu,1)}\right)+\mathrm{Re}\left(H_k^{(\mu,2)}\right)
			 & 
			\mathrm{Im}\left(H_k^{(\mu,1)}\right)-\mathrm{Im}\left(H_k^{(\mu,2)}\right)\\
			-\mathrm{Im}\left(H_k^{(\mu,1)}\right)-\mathrm{Im}\left(H_k^{(\mu,2)}\right)&
			\mathrm{Re}\left(H_k^{(\mu,1)}\right)-\mathrm{Re}\left(H_k^{(\mu,2)}\right)
		\end{array}\right], \\
		\Delta_{\mu}\widetilde{\mathbb U}_j^{(2)\mathrm R}-\widetilde{\mathbb 
			A}_j^{(2)\mathrm R}=\delta_{hl}
		\left[\begin{array}{ll}
			\mathrm{Re}\left(H_k^{(\lambda,1)}\right)+\mathrm{Re}\left(H_k^{(\lambda,2)}\right)
			& 
			\mathrm{Im}\left(H_k^{(\lambda,1)}\right)-\mathrm{Im}\left(H_k^{(\lambda,2)}\right)\\
			-\mathrm{Im}\left(H_k^{(\lambda,1)}\right)-\mathrm{Im}\left(H_k^{(\lambda,2)}\right)&
			\mathrm{Re}\left(H_k^{(\lambda,1)}\right)-\mathrm{Re}\left(H_k^{(\lambda,2)}\right)
		\end{array}\right],
	\end{split}
	\right.
\end{equation}
where the coefficient matrices are 
\begin{equation*}
\Delta_{\lambda}=\left[\begin{array}{ll}
	0 & \lambda_{i_j,j} \\
	-\lambda_{i_j,j} & 0
\end{array}\right], \Delta_{\mu}=\left[\begin{array}{ll}
0 & \mu_{i_j,j} \\
-\mu_{i_j,j} & 0
\end{array}\right], 
\end{equation*}
and $\mathrm{Re}(\cdot), \mathrm{Im}(\cdot)$ denote the real part and the imaginary 
part of a complex number, respectively. Consequently, we can get $\widetilde{\mathbb 
U}_j^{(2)\mathrm R}$ and $\widetilde{\mathbb A}_j^{(2)\mathrm R}$ by solving the 
above linear system completely in real arithmetic. 

Similarly, we consider the $2\times 2$ block matrix $\widetilde{\mathbb N}_j^{(2)}$ 
of $\widetilde{\mathbb N}_j$. It follows from (\ref{N-exp}) that there holds 
\begin{equation*}
	\widetilde{\mathbb N}_j^{(2)}=\left[
	\begin{array}{cc}
		\widetilde{\mathbb N}_j^{(h,l)}&\widetilde{\mathbb N}_j^{(h,l+1)}\\
		\widetilde{\mathbb N}_j^{(h+1,l)}&\widetilde{\mathbb N}_j^{(h+1,l+1)}
	\end{array}
	\right]=\frac{\delta_{hl}}{\phi_j}\left[
	\begin{array}{cc}
		H_k^{(\lambda,1)}&H_k^{(\lambda,2)}\\
		\overline{H_k^{(\lambda,2)}}&\overline{H_k^{(\lambda,1)}}
	\end{array}
	\right]. 
\end{equation*}
With the transformation $\widetilde{\mathbb N}_j^{(2)\mathrm 
R}=J^*\widetilde{\mathbb N}_j^{(2)}J$, one can get 
\begin{equation*}
	\widetilde{\mathbb N}_j^{(2)\mathrm 
	R}=\frac{\delta_{hl}}{\phi_j}\left[\begin{array}{ll}
		\mathrm{Re}\left(H_k^{(\lambda,1)}\right)+\mathrm{Re}\left(H_k^{(\lambda,2)}\right)
		& 
		\mathrm{Im}\left(H_k^{(\lambda,1)}\right)-\mathrm{Im}\left(H_k^{(\lambda,2)}\right)\\
		-\mathrm{Im}\left(H_k^{(\lambda,1)}\right)-\mathrm{Im}\left(H_k^{(\lambda,2)}\right)&
		\mathrm{Re}\left(H_k^{(\lambda,1)}\right)-\mathrm{Re}\left(H_k^{(\lambda,2)}\right)
	\end{array}\right]. 
\end{equation*}
The $2\times 1$ and $1\times 2$ blocks of $\widetilde{\mathbb B}_1, 
\widetilde{\mathbb C}_k$, respectively, are defined as 
\begin{equation*}
	\widetilde{\mathbb B}_1^{(2)}=\left[
	\begin{array}{c}
	\widetilde{\mathbb B}_1^{(h)} \\
	\widetilde{\mathbb B}_1^{(h+1)}
	\end{array}
	\right]=\left[
	\begin{array}{c}
		\widetilde{\mathbb B}_1^{(h)} \\
		\overline{\widetilde{\mathbb B}_1^{(h)}}
	\end{array}
	\right], 
\end{equation*}
\begin{equation*}
	\widetilde{\mathbb C}_k^{(2)}=\left[
	\begin{array}{c}
		\widetilde{\mathbb C}_k^{(l)} \quad
		\widetilde{\mathbb C}_k^{(l+1)}
	\end{array}
	\right]=\left[
	\begin{array}{c}
		\widetilde{\mathbb C}_k^{(l)}\quad 
		\overline{\widetilde{\mathbb C}_k^{(l)}}
	\end{array}
	\right]. 
\end{equation*}
By defining the transformation $\widetilde{\mathbb B}_1^{(2)\mathrm 
R}=J^*\widetilde{\mathbb B}_1^{(2)}$ and $\widetilde{\mathbb C}_k^{(2)\mathrm 
R}=\widetilde{\mathbb C}_k^{(2)}J$, it yields 
\begin{equation*}
	\widetilde{\mathbb B}_1^{(2)\mathrm R}=\sqrt{2}\left[
	\begin{array}{c}
		\mathrm{Re}\left(\widetilde{\mathbb B}_1^{(h)}\right)\\
		-\mathrm{Im}\left(\widetilde{\mathbb B}_1^{(h)}\right)
	\end{array}
	\right], \widetilde{\mathbb C}_k^{(2)\mathrm R}=\sqrt{2}\left[
	\begin{array}{c}
		\mathrm{Re}\left(\widetilde{\mathbb C}_k^{(l)}\right)\quad 
		\mathrm{Im}\left(\widetilde{\mathbb C}_k^{(l)}\right)
	\end{array}
	\right].
\end{equation*}
As a result, the real-valued counterparts associated with the main terms 
$\widetilde{\mathbb{U}}_{j}, \widetilde{\mathbb{A}}_{j}, \widetilde{\mathbb{N}}_{j}, 
\widetilde{\mathbb{B}}_{1}$ and $\widetilde{\mathbb{C}}_{k}$ are given by 
\begin{equation*}
	\begin{split}
	\widetilde{\mathbb{U}}_{j}^{\mathrm R}=(I_{\bar{\mathcal 
	N}_j/2}\otimes J^*)\widetilde{\mathbb{U}}_{j}(I_{\mathcal 
	N_j/2}\otimes J),\quad &\widetilde{\mathbb{A}}_{j}^{\mathrm R}=(I_{\bar{\mathcal 
		N}_j/2}\otimes J^*)\widetilde{\mathbb{A}}_{j}(I_{\mathcal 
	N_j/2}\otimes J),\\
	\widetilde{\mathbb{N}}_{j}^{\mathrm R}=(I_{\bar{\mathcal 
			N}_{j+1}/2}\otimes J^*)\widetilde{\mathbb{N}}_{j}(I_{\mathcal 
		N_j/2}\otimes J), \quad &\widetilde{\mathbb{B}}_{1}^{\mathrm 
		R}=(I_{\bar{\mathcal 
			N}_1/2}\otimes J^*)\widetilde{\mathbb{B}}_{1}, \quad
			\widetilde{\mathbb{C}}_{k}^{\mathrm 
			R}=\widetilde{\mathbb{C}}_{k}(I_{\mathcal 
		N_k/2}\otimes J), 
	\end{split}
\end{equation*}
where $I$ is the identity matrix, and the subscript denotes its dimension. In 
practice, we can replace the main terms $\widetilde{\mathbb{U}}_{j}, 
\widetilde{\mathbb{A}}_{j}, \widetilde{\mathbb{N}}_{j}, \widetilde{\mathbb{B}}_{1}$ 
and $\widetilde{\mathbb{C}}_{k}$ in Algorithm \ref{data_BT} by  
$\widetilde{\mathbb{U}}_{j}^{\mathrm R}, \widetilde{\mathbb{A}}_{j}^{\mathrm R}, 
\widetilde{\mathbb{N}}_{j}^{\mathrm R}, \widetilde{\mathbb{B}}_{1}^{\mathrm R}$ and 
$\widetilde{\mathbb{C}}_{k}^{\mathrm R}$, and implement the data-driven BT algorithm 
completely in real arithmetic, thereby leading to real-valued reduced models. 

\begin{remark}
	It follows from (\ref{coup-sylvster}) that $\widetilde{\mathbb U}_j^{(2)\mathrm 
	R}$ satisfies a Sylvester equation, and thereby $\widetilde{\mathbb 
	U}_j^{\mathrm R}$ also is the solution of a large-scale Sylvester equation. In 
	\cite{Hamadi2023}, the low-rank approximate solution of Sylvester equations is 
	exploited to accelerate the execution of Loewner frameworks. Due to 
	the full-rank right-hand side, the accurate approximation to $\widetilde{\mathbb 
	U}_j^{\mathrm R}$ may not be available by solving the associated Sylvester 
	equation approximately, and we implement our approach by assembling 
	$\widetilde{\mathbb U}_j^{\mathrm R}$ explicitly in the simulation. 
\end{remark}

\section{Numerical results}\label{sec:sec-4}
In this section, we perform numerical simulation to illustrate the effectiveness of 
our approach. The proposed scheme is carried out via Matlab (R2021a) on a laptop 
with Intel(R) Core(TM) i5-8265U and 8 GB RAM.
The reduced models produced by Algorithm 1 and Algorithm 2 are referred as BT and 
DKBBT, respectively. We simply
use the same number of nodes for the quadrature approximation to $P_{jj}$ and 
$Q_{jj}$ about $\omega_j$ for $j=1, 2, \cdots, k$. The 
Matlab function \texttt{logspace} is used to produce the logarithmically spaced 
points $\lambda_{{i_j},j}$ and $\mu_{i_j,j}$, and the $k$-th
transfer function is evaluated at $\mathrm i\lambda_{{i_j},j}, \mathrm i\mu_{i_j,j} 
(j=1, 2, \cdots, k)$ to collect the sample data. The quadrature weights are 
chosen by the trapezoid quadrature rule in the approximation. More details on other 
representative quadrature rules can be found in \cite{Gosea2022}.

\begin{figure}[htb]
	\centering
	\begin{minipage}{0.40\textwidth} 
		\centering 
		\includegraphics[width=0.98\textwidth]{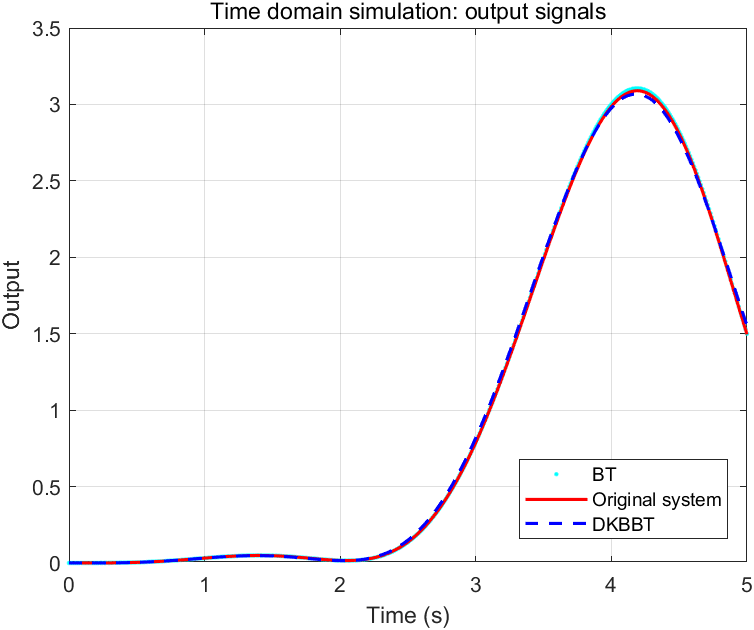}
	\end{minipage}
	\begin{minipage}{0.40\textwidth} 
		\centering 
		\includegraphics[width=0.98\textwidth]{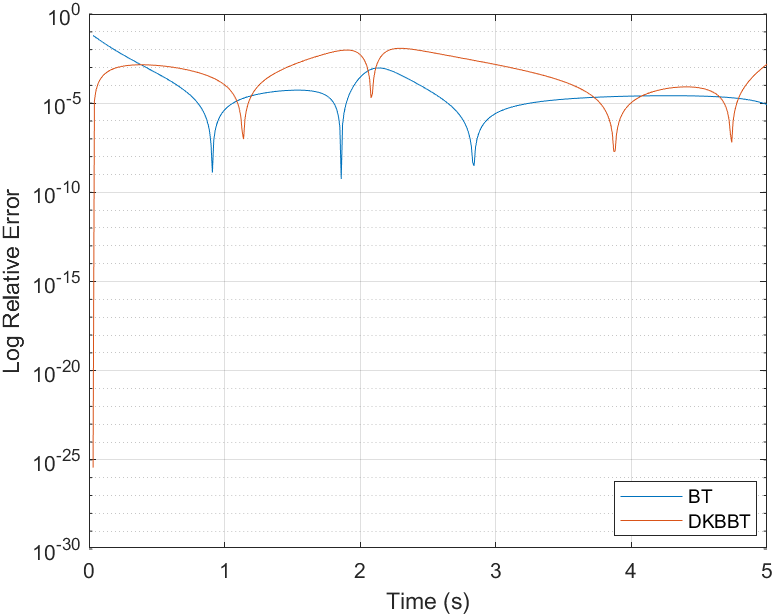}
	\end{minipage}
	\caption{Left: the responses for $u(t)=t\cos(t)$; Right: 
		relative errors for each method}
	\label{fig1}
\end{figure}

We consider  the K-power system composed of two subsystems given in \cite{Jin2023}. 
The order of each subsystem is $n_1=n_2=300$, and the order of K-power 
system is $n=600$. The system is determined by the following coefficient matrices
\begin{equation*}
A_{1}=\left[\begin{array}{lllll}
	-10 & 2 &  &  &    \\
	7 & -10 & 2 &  &    \\
	 & \ddots & \ddots &\ddots &    \\
	 &  &  7 & -10 & 2 \\
	 &  &  &  7 & -10
\end{array}\right],		
A_{2}=\left[\begin{array}{lllll}
	-5 & 2 &  &  &    \\
	2 & -5 & 2 &  &    \\
	& \ddots & \ddots &\ddots &    \\
	&  &  2 & -5 & 2 \\
	&  &  &  2 & -5
\end{array}\right],		
\end{equation*}		
\begin{equation*}	
N_{11}=\left[\begin{array}{lllll}
	2 & 1 &  &  &    \\
	-1 & 2 & 1 &  &    \\
	& \ddots & \ddots &\ddots &    \\
	&  &  -1 & 2 & 1 \\
	&  &  &  -1 & 2
\end{array}\right],	
B_{1}=\left[\begin{array}{l}
	1 \\
	1\\
	\vdots \\
	1\\
	1
\end{array}\right],
C_{2}^{\top}=\left[\begin{array}{l}
	0\\
	0\\
	\vdots\\
	0\\
	1
\end{array}\right].
\end{equation*}			
With the reduced order $r_1=r_2=25$, two reduced models are generated by Algorithm 1 
and 2. Figure \ref{fig1} and Figure \ref{fig2} depict the time responses and the 
associated relative errors of two 
reduced models when the system is impulsed by the inputs $u(t)=t\cos(t)$ and 
$u(t)=\sin(0.5t)\mathrm{e}^{-0.5t}$, respectively. The dynamical behavior of the 
original systems is well approximated by the two reduced models, and we can hardly 
distinguish them clearly for the depiction. The relative error shows that the 
data-driven BT and the standard BT exhibit almost the same approximation accuracy in 
this example, which agrees with our expectation. Note that the slight discrepancy 
shown in the relative depiction comes from the round-off in the execution of 
Algorithm 1 and Algorithm 2.

\begin{figure}[htb]
	\centering
	\begin{minipage}{0.40\textwidth} 
		\centering 
		\includegraphics[width=0.98\textwidth]{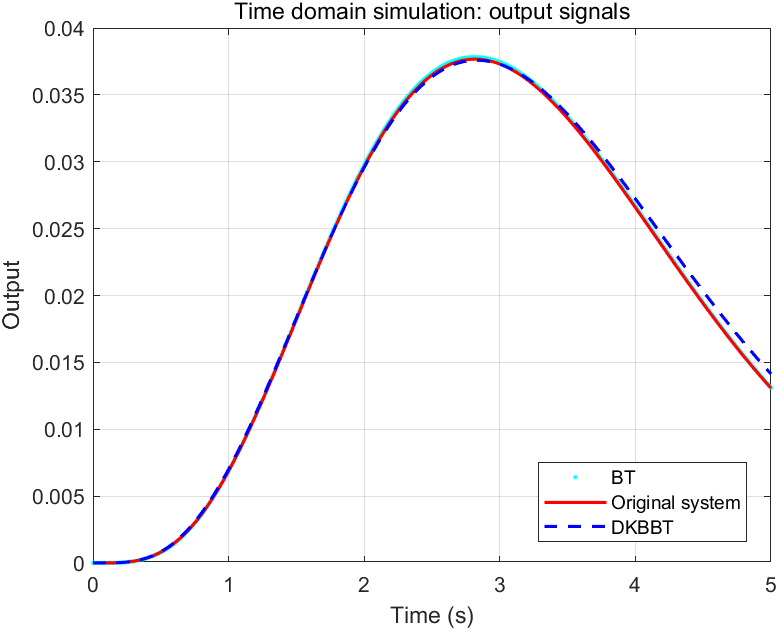}
	\end{minipage}
	\begin{minipage}{0.40\textwidth} 
		\centering 
		\includegraphics[width=0.98\textwidth]{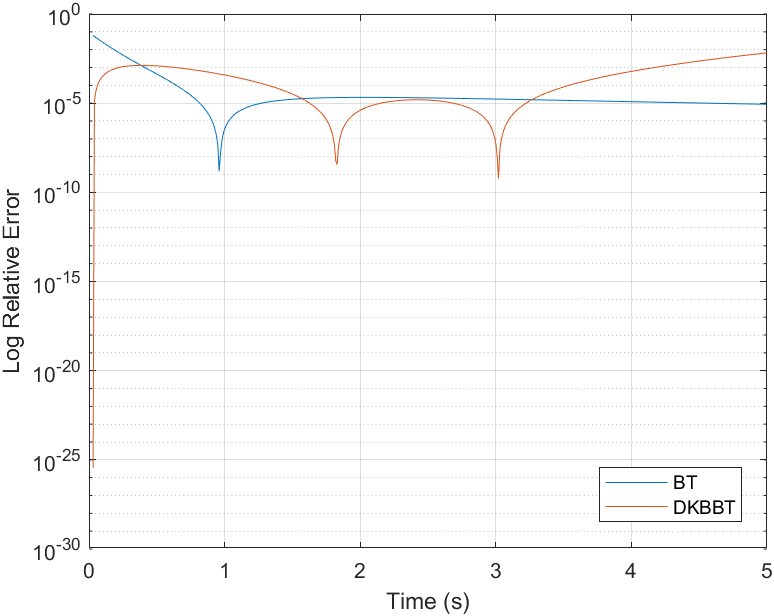}
	\end{minipage}
	\caption{Left: the responses for $u(t)=\sin(0.5t)\mathrm{e}^{-0.5t}$; Right: 
		relative errors for each method}
	\label{fig2}
\end{figure}

\section{Conclusions}\label{sec:sec-5}
	We have presented a nonintrusive BT procedure for K-power bilinear systems by 
	using the evaluations of transfer function. The numerical quadrature rule 
	provides a low-rank approximation to Grimians of the systems, and the explicit 
	relationship between the low-rank execution of BT procedure and the measurements 
	of systems pave the way for producing reduced models directly in a nonintrusive 
	manner. The execution of the proposed approach in real arithmetic is discussed 
	in detail, and the real-valued reduced models can be generated naturally, which 
	facilitates the application of our approach a lot in practice. The simulation 
	results indicate that our method can reproduce faithfully the performance of the 
	standard BT approach for K-power systems. 	

\appendix
\setcounter{figure}{0}

	\addcontentsline{toc}{section}{Reference}
	\markboth{Reference}{}
	\bibliographystyle{elsarticle-num}
	\bibliography{reference}

@article{Hamadi2023,
  author  = {Hamadi, M.A. and Jbilou, K. and Ratnani, A.},
  journal = {Journal of Scientific  Computing},
  title   = {A Data-Driven {Krylov} Model Order Reduction for Large-Scale Dynamical Systems},
  year    = {2023},
  volume  = {95},
  number  = {2},
  pages   = {1-27},
  doi     = {}
}

@book{Benner2017book,
  author    = {Benner, Peter},
  title     = {System Reduction for Nanoscale IC Design},
  year      = {2017},
  publisher = {Springer Nature, Switzerland}
}

@inproceedings{Ramaswamy2000,
  author    = {Ramaswamy, D. and White, J.},
  booktitle = {Technical Proceedings of the Fourth International Conference on Modeling and Simulation of Microsystems},
  title     = {Automatic generation of small-signal dynamic macromodels from {3-D} simulation},
  year      = {2000},
  volume    = {},
  number    = {},
  pages     = {27-30}
}

@article{Su1991,
  author  = {Su, T.J. and Raig, R.R.},
  journal = {Journal of Guidance, Control, and Dynamics},
  title   = {Model reduction and control of flexible structures using {Krylov} vectors},
  year    = {1991},
  volume  = {14},
  number  = {2},
  pages   = {260-267}
}

@book{Antoulasbook2005,
  author    = {Antoulas, A.C.},
  title     = {Approximation of Large-Scale Dynamical Systems},
  year      = {2005},
  publisher = {SIAM, Philadelphia}
}

@article{Baur2014,
  author  = {Baur, Ulrike and Benner, Peter and Feng, Lihong},
  journal = {Archives of Computational Methods in Engineering},
  title   = {Model Order Reduction for Linear and Nonlinear Systems: A System-Theoretic Perspective},
  year    = {2014},
  volume  = {21},
  number  = {},
  pages   = {331-358}
}

@article{Pan2013,
  author  = {Pan, Xiaoda and Zhu, Hengliang and Yang, Fan and Zeng, Xuan},
  journal = {Communications in Computational Physics},
  title   = {Subspace Trajectory Piecewise-Linear Model Order Reduction for Nonlinear Circuits},
  year    = {2013},
  volume  = {14},
  number  = {3},
  pages   = {639-663}
}

@article{Chaturantabut2010,
  author  = {Chaturantabut, Saifon and Sorensen, D.C.},
  journal = {SIAM Journal on Scientific Computing},
  title   = {Nonlinear model reduction via discrete empirical interpolation},
  year    = {2010},
  volume  = {32},
  number  = {5},
  pages   = {2737-2764}
}

@article{Gu2011,
  author  = {Gu, Chenjie},
  journal = {IEEE Transactions on Computer-Aided Design of Integrated Circuits and Systems},
  title   = {A projection-based nonlinear model order reduction approach using quadratic-linear representation of nonlinear systems},
  year    = {2011},
  volume  = {30},
  number  = {},
  pages   = {1307-1320}
}

@book{Rughbook1981,
  author    = {Rugh, W.J.},
  title     = {Nonlinear System Theory},
  year      = {1981},
  publisher = {Johns Hopkins University Press, Baltimore}
}

@article{Zhang2002,
  author  = {Zhang, Liqian and Lam, James},
  journal = {Automatica},
  title   = {On {$\mathcal H_2$} model reduction of bilinear systems},
  year    = {2002},
  volume  = {38},
  number  = {},
  pages   = {205-216}
}

@article{Bennerbil2012,
  author  = {Benner, Peter and Breiten, Tobias},
  title   = {Interpolation-based {$\mathcal H_2$}-model reduction of bilinear control systems},
  journal = {SIAM Journal on Matrix Analysis and Applications},
  year    = {2012},
  volume  = {33},
  number  = {3},
  pages   = {869-885}
}

@article{Flagg2015,
  author  = {Flagg, Garret and Gugercin, Serkan},
  title   = {Multipoint {Volterra} series interpolation and {$\mathcal H_2$} optimal model reduction of bilinear systems},
  journal = {SIAM Journal on Matrix Analysis and Applications},
  year    = {2015},
  volume  = {36},
  number  = {2},
  pages   = {549-579}
}

@inproceedings{Hsu1983,
  author    = {Hsu, C.S. and Desai, U.B. and Crawley, C.A.},
  booktitle = {Proceedings of the 22nd IEEE Conference on Decision and Control, San Antonio, TX},
  title     = {Realization algorithms and approximation methods of bilinear systems},
  year      = {1983},
  volume    = {},
  number    = {},
  pages     = {783-788}
}

@article{Bennerbil2011,
  author  = {Benner, Peter and Damm, T.},
  title   = {Lyapunov equations, energy functionals, and model order reduction of bilinear and stochastic systems},
  journal = {SIAM Journal on Control and Optimization},
  year    = {2011},
  volume  = {49},
  number  = {},
  pages   = {686-711}
}

@article{Wang2012,
  author  = {Wang, Xiaolong and Jiang, Yaolin},
  title   = {Model reduction of bilinear systems based on {Laguerre} series expansion},
  journal = {Journal of the Franklin Institute},
  year    = {2012},
  volume  = {349},
  number  = {},
  pages   = {1231-1246}
}

@inproceedings{Baiyat1993,
  author    = {Al-baiyat, S.A. and Bettayet, M.},
  booktitle = {Proceedings of the 32nd Conferences on Decision and Control, San Antonio, TX},
  title     = {A New Model Reduction Scheme for {K-power} Bilinear Systems},
  year      = {1993},
  volume    = {},
  number    = {},
  pages     = {22-27}
}

@article{Wang2014,
  author  = {Wang, Xiaolong and Jiang, Yaolin},
  title   = {On mdoel reduction of {K-power} bilinear systems},
  journal = {International Journal of Systems Science},
  year    = {2014},
  volume  = {45},
  number  = {9},
  pages   = {1978-1990}
}

@article{Qi2021,
  author  = {Qi, Zhenzhong and Jiang, Yaolin and Xiao, Zhihua},
  title   = {Dimension reduction for {K-power} bilinear systems using orthogonal polynomials and {Arnoldi} algorithm},
  journal = {International Journal of Systems Science},
  year    = {2021},
  volume  = {52},
  number  = {10},
  pages   = {2048-2063}
}

@article{Zhang2024,
  author  = {Zhang, Yansong and Xiao, Zhihua and Jiang, Yaolin},
  title   = {Finite-time model order reduction for {K-power} bilinear systems via shifted {Legendre} polynomials},
  journal = {Computational and Applied Mathematics},
  year    = {2024},
  volume  = {43},
  number  = {},
  pages   = {158}
}

@article{Gosea2022,
  author  = {Gosea, I. V. and Gugercin, S. and Beattie , C.},
  title   = {Data-driven balancing of linear dynamical systems},
  journal = {SIAM Journal on Scientific Computing},
  year    = {2022},
  volume  = {44},
  number  = {1},
  pages   = {A554–A582}
}

@article{Padhi2025,
  author  = {Padhi, Reetish and Gosea, Ion Victor and Duff, Igor Pontes and Gugercin, Serkan},
  title   = {Data-Driven Balancing Formulation for Linear Systems With Quadratic Outputs},
  journal = {IEEE Control Systems Letters},
  year    = {2025},
  volume  = {9},
  number  = {},
  pages   = {2843-2848}
}

@article{Wang2025,
  author  = {Wang, Xiaolong and Yang, Xuerong and Wang, Xiaoli and Jiang, Yaolin},
  title   = {Data-driven balanced truncation for second-order systems via the approximate {Gramians}},
  journal = {Numerical Algorithms},
  year    = {2025},
  volume  = {online},
  number  = {},
  pages   = {https://doi.org/10.1007/s11075-025-02253-z}
}

@article{Bhattacharjee2025,
  author  = {Bhattacharjee, Debraj and Moreschini, Alessio and Astolfi, Alessandor},
  title   = {Signal generator agnostic moment metching},
  journal = {IEEE Tansactioins on Automatic Control},
  year    = {2025},
  volume  = {70},
  number  = {1},
  pages   = {7493-7508}
}

@article{Geelen2023,
  author  = {Geelen, Rudy and Wright, Stephen and Willcox, Karen},
  title   = {Operator inference for non-intrusive model reduction with quadratic manifolds},
  journal = {Computer Methods in Applied Mechanics and Engineering},
  year    = {2023},
  volume  = {403},
  number  = {},
  pages   = {115717}
}

@article{Ionita2014,
  author  = {Ionita, A.C. and Antoulas, A.C.},
  title   = {Data-driven paramterized model reduction in the {Loewner} framework},
  journal = {SIAM Journal on Scientific Computing},
  year    = {2014},
  volume  = {36},
  number  = {3},
  pages   = {A984-A1007}
}

@article{Huhn2023,
  author  = {Huhn, Quincy A. and Tano, Mauricio E. and Ragusa, Jean C. and Choi, Youngsoo},
  title   = {Parametric dynamic mode decomposition for reduced order modeling},
  journal = {Journal of Computational Physics},
  year    = {2023},
  volume  = {475},
  number  = {},
  pages   = {111852}
}

@article{Burohman2023,
  author  = {Burohman, Azka Muji and Besselink, Bart and Scherpen, J.M.A. and Camlibel, M.K.},
  title   = {From Data to Reduced-Order Models via Generalized Balanced Truncation},
  journal = {IEEE Tansactioins on Automatic Control},
  year    = {2023},
  volume  = {68},
  number  = {10},
  pages   = {6160-6175}
}

@article{Jin2023,
  author  = {Jin, H. and Xiao, Zhihua and Song, Qiuyan and Qi, Zhenzhong},
  title   = {Structure-preserving model order reduction for {K-power} bilinear systems via {Laguerre} functions},
  journal = {International Journal of Systems Science},
  year    = {2023},
  volume  = {54},
  number  = {8},
  pages   = {1648-1660}
}

\end{document}